\documentclass[preprint,12pt]{elsarticle}

\usepackage{enumitem}

 \usepackage{graphicx}

\usepackage{amsmath}

\usepackage{amssymb}
\usepackage{amsthm}

\usepackage{mathtools}
\mathtoolsset{showonlyrefs=true}
\mathtoolsset{centercolon}

\theoremstyle{plain}
\newtheorem{theorem}{Theorem}[section]
\newtheorem{corollary}[theorem]{Corollary}
\newtheorem{proposition}[theorem]{Proposition}
\newtheorem{lemma}[theorem]{Lemma}

\theoremstyle{definition}
\newtheorem{definition}[theorem]{Definition}

\newtheorem{example}[theorem]{Example}

\theoremstyle{remark}
\newtheorem{remark}[theorem]{Remark}

\newproof{pf}{proof}

\usepackage{soul,color}
\setstcolor{red}
\setul{0.8ex}{0.2ex}
\usepackage{proofread}
\ifdefined\verproofread
\else
	\noproofreadmark
\fi

\newcommand{\norm}[1]{\lVert{#1}\rVert}

\newcommand{\abs}[1]{\lvert #1 \rvert}

\DeclareMathOperator\supp{supp}
\DeclareMathOperator{\interior}{int}

\ifdefined\usehyperref
\usepackage{hyperref}
\hypersetup{
pdfborder = {0 0 0},
linkcolor = {black},
citecolor = {black},
urlcolor= {black}}
\fi

\ifdefined\refcheck
\usepackage{refcheck}
\fi

\begin{document}

\begin{frontmatter}

\title{A Limit Formula for Joint Spectral Radius with $p$-radius of
Probability Distributions}

\author[ttu]{M.~Ogura\corref{cor}}
\ead{masaki.ogura@ttu.edu}

\author[ttu]{C.~F.~Martin}
\ead{clyde.f.martin@ttu.edu}

\cortext[cor]{Corresponding author}

\address[ttu]{Department of Mathematics and Statistics, Texas Tech
University, Broadway and Boston, Lubbock, TX 79409-1042, USA}


\begin{abstract}
In this paper we show a characterization of the joint spectral radius
of a set of matrices as the limit of the $p$-radius of an associated
probability distribution when $p$ tends to $\infty$. Allowing the set
to have infinitely many matrices, the obtained formula extends the
results in the literature. Based on the formula, we then present a
novel characterization of the stability of switched linear systems for
an arbitrary switching signal via the existence of stochastic Lyapunov
functions of any higher degrees. Numerical examples are presented to
illustrate the results.
\end{abstract}


\end{frontmatter}

\section{Introduction}

The joint spectral radius of a set of matrices, originally introduced
in the short note~\cite{Rota1960}, is a natural extension of the
spectral radius of a single matrix and has found various applications
in, for example, wavelet theory, functional analysis, and systems and
control theory (see the monograph \cite{Jungers2009} for detail). This
wide range of applications has motivated many authors to study the
computation of joint spectral radius. Though even the approximation of
joint spectral radius is in general an NP-hard
problem~\cite{Tsitsiklis1997a}, there are now a vast amount of
efficient methods for the approximation of joint spectral radius
\cite{Blondel2005,Parrilo2008,Protasov2010} and also their
implementations on mathematical softwares~\delled{{\cite{Jungers2011t}}}\delwsed\added{\cite{Vankeerberghen2014a}}.

The result \cite{Blondel2005} by Blondel and Nesterov is of a
particular theoretical interest because it characterizes joint
spectral radius as the limit of another joint spectral characteristics
called $L^p$-norm joint spectral radius when $p$ tends to $\infty$.
Given a finite set~$\mathcal M = \{A_1, \dotsc, A_N\}$ of real and
square matrices of a fixed dimension and a parameter~$p\geq 1$, the
{\it $L^p$-norm joint spectral radius} ({\it $p$-radius} for short) of
$\mathcal M$ is defined by
\begin{equation} \label{eq:classicGJSR}
\rho_{p,\mathcal M} 
:= 
\lim_{k\to\infty}\biggl(N^{-k}\sum_{i_1, \dots, i_k\in \{1, \dotsc, N  \}}
\norm{A_{i_k} \dotsm A_{i_1}}^p\biggr)^{1/kp}, 
\end{equation}
where $\norm{\cdot}$ denotes any matrix norm. Firstly
introduced~\cite{Jia1995,Wang1996} for $p=1$ and then
extended~\cite{Protasov1997} for a general $p$, $L^p$-norm joint
spectral radius has found many applications in various areas of
applied mathematics (see \cite{Jungers2011b} and references therein).
In particular $p$-radius has an application to the stability theory of
stochastic switched systems~\cite{Jungers2010,Ogura2012b,Ogura2013},
which is a dynamical system whose structure randomly experiences
abrupt changes~\cite{Lin2009,Shorten2007}.

Recently this ``original'' version of $L^p$-norm joint spectral radius
was extended to probability distributions~\cite{Ogura2012b}. Roughly
speaking, the extension makes it possible to consider the $p$-radius
of \changed{an infinite set of matrices}{a set of {\it
infinitely} many matrices} and is useful when, for example, one wants
to study the stability of a stochastic switched system with infinitely
many subsystems that naturally arise as a result of uncertainty in
modeling of dynamical systems. Being an extension, the $p$-radius of
distributions inherits~\cite{Ogura2012b} from the $p$-radius of sets
of matrices the characterization~\cite{Protasov1997} as the spectral
radius of a matrix. Though the characterization is valid only either
when $p$ is an even integer or when matrices in~$\mathcal M$ leave a
common proper cone invariant, it still covers several interesting
cases that appear in the stability analysis of stochastic switched
linear systems. Then it is natural to expect that the other properties
of the $p$-radius of sets of matrices can be extended to the
$p$-radius of distributions.

In this paper we show that the characterization by Blondel and
Nesterov~\cite{Blondel2005} is still valid when we use the $p$-radius
of probability distributions. This extension in particular circumvents
the finiteness limitation of the original characterization. Since the
proof for the original result relies on the finiteness of the number
of matrices, it cannot be directly applied to the current setting.
Instead, our proof extensively utilizes \changed{nearly most unstable
trajectories of an associated switched linear system and is
independent of the number of possible matrices}{so-called cone linear
absolute norms \cite{Seidman2005} and the approximation of a given set
of possibly infinitely many matrices by \change{its subset}{subsets}
having a certain uniformity property}.

As a theoretical application of the characterization of joint spectral
radius, we will discuss the stability of switched linear systems. We
will present a novel characterization of the stability of a switched
linear system for an arbitrary switching signal with a so-called
stochastic Lyapunov function~\cite{Bertram1959,Ahmadi1979,Feng1992},
which is a positive definite functional whose value decreases along
the trajectory of the switched linear system in expectation. The
characterization in particular deduces the existence of stochastic
Lyapunov functions from stability and hence is a variant of the
converse Lyapunov theorems~\cite{Dayawansa1999,Molchanov1989} in
systems and control theory. The construction of stochastic Lyapunov
functions is also investigated.

This paper is organized as follows. After preparing necessary
notations in Section~\ref{sec:math}, in Section~\ref{sec:JSCs} we give
a brief overview of the joint spectral radius of sets of matrices and
the $L^p$-norm joint spectral radius of probability distributions.
Then Section~\ref{sec:JSRchar} gives the characterization of joint
spectral radius as the limit of $L^p$-norm joint spectral radius. In
Section~\ref{sec:ConvLypThm} we discuss the application of the
characterization to the stability theory of switched linear systems.

\section{Mathematical Preliminaries} \label{sec:math}

\added{Let $\mathbb{R}_+$ denote the set of nonnegative real numbers.}
For $x\in \mathbb{R}^n$ its Euclidean norm is denoted by
$\norm{x}$\added{, if not explicitly stated otherwise}. \delled{The usual
inner product in $\mathbb{R}^n$ is denoted by $(\cdot,
\cdot)$.}{\delwsed}For a real matrix~$A$ its maximal singular value is
denoted by~$\norm{A}$. If $A$ is square then its spectral radius is
denoted by $\rho(A)$. When $A$ is symmetric and negative semidefinite
we write $A\preceq 0$. Let $\mathcal M \subset \mathbb{R}^{n\times
n}$. The interior and the boundary of $\mathcal M$ are denoted by
$\interior \mathcal M$ and $\partial \mathcal M$, respectively. The
distance between $A$ and $\mathcal{M}$ is defined by $d(A, \mathcal M)
:= \inf_{M\in\mathcal M}\norm{A-M}$.

Let $\Omega$ be a probability space with a probability measure~$\mu$.
The support of~$\mu$, denoted by~$\supp \mu$, is defined as the closed
set such that $\mu((\supp \mu)^c) = 0$ and, if $G$ is open and $G \cap
{(\supp \mu)} \neq \emptyset$, then $\mu(G \cap \supp \mu) > 0$.
Dirac's delta distribution on~$x \in \Omega$ is denoted by $\delta_x$.
For an integrable random variable~$X$ on~$\Omega$ its expected value
is denoted by $E[X]$. \delled{The next proposition is fundamental.}

\delxed{\begin{lemma}[\cite{Lang1986}]\label{lem:nullset.old}
\delled{The boundary of a convex set in $\mathbb{R}^{n\times n}$ is a
null set with respect to the Lebesgue measure.}
\end{lemma}}


\subsection{Proper Cones}

A subset~$K\subset \mathbb{R}^n$ is called a {cone} if $K$ is closed
under multiplication by nonnegative numbers. The cone is said to be
solid if it possesses a nonempty interior. We say that a cone is
pointed if $x, -x \in K$ implies $x=0$. We say that $K$ is proper if
it is closed, convex, solid, and pointed. Let $K \subset \mathbb{R}^n$
be a proper cone.  A matrix $A\in \mathbb{R}^{n\times n}$ is said to
leave $K$ invariant\delwsed\delled{or to be $K$-nonnegative}, written
$A\geq^K 0$, if $AK\subset K$. The set of all \added{real} matrices
leaving $K$ invariant is denoted by $\pi(K)$ or simply by $\pi$. Let
$B \in \mathbb{R}^{n\times n}$. By $A\geq^K B$ we mean $A-B \geq^K 0$.
A set~$\mathcal M \subset \mathbb{R}^{n\times n}$ is said to leave $K$
invariant if any $A\in \mathcal M$ leaves $K$ invariant. $A$ is said
to be $K$-positive if $A(K - \{0\}) \subset \interior K$ and we write
$A>^K 0$. \added{We understand $A >^K B$ in the obvious way.} It is
known that~\cite[p.~16]{Berman1979}
\begin{equation}\label{eq:intpi(K)}
\interior  \pi = 
\{A\in \mathbb{R}^{n\times n}: A>^K 0 \}.
\end{equation} 
\added{Also we can show the next lemma. 
\begin{lemma}\label{lem:nullset}
The boundary $\partial \pi$ is a null set with respect to the Lebesgue
measure.
\end{lemma}
\begin{proof}
In general, the boundary of a convex set in $\mathbb{R}^{n\times n}$
is a null set with respect to the Lebesgue
measure~\cite[Theorem~1]{Lang1986}. This proves the claim because
$\pi$ is clearly convex.
\end{proof}}

\added{A class of norms called cone linear absolute norms (see, e.g.,
\cite{Seidman2005}) plays an important role in this paper.} A norm
$\norm{\cdot}$ on $\mathbb{R}^n$ is said to be cone
absolute~\cite{Seidman2005} with respect to a proper cone~$K$ if, for
every $x\in\mathbb{R}^n$,
\begin{equation}\label{eq:ConeAbs}
\norm{x} = \inf_{\substack{v, w\in K\\x=v-w}}\norm{v+w}. 
\end{equation}
Also we say that $\norm{\cdot}$ is cone linear with respect to $K$ if
there exists $f$ in the dual cone $K^* := \{ f\in\mathbb{R}^n : f^\top
x \geq 0 \text{ for every $x\in K$}  \}$ such that
\begin{equation} \label{eq:ConeLin}
\norm{x} = f^\top x
\end{equation}
for every $x\in K$. A norm that is cone linear and cone absolute with
respect to a \delled{common}{\delwsed}proper cone is said to be {cone
linear absolute}. It is known~\cite{Seidman2005} that every $f\in
\interior (K^*)$ yields a cone linear absolute norm~$\norm{\cdot}_f$
determined by \eqref{eq:ConeAbs} and~\eqref{eq:ConeLin}. The
norm~$\norm{\cdot}_f$ induces a norm on $\mathbb{R}^{n\times n}$ as
\begin{equation}\label{eq:normA_f}
\norm{A}_f 
:= 
\sup_{x\in\mathbb{R}^n}\frac{\norm{Ax}_f}{\norm{x}_f}.
\end{equation}
\added{When $f$ is irrelevant we simply denote $\norm{\cdot}_f$ by
$\norm{\cdot}$.} Some useful properties of this norm are quoted from
\cite{Seidman2005} in the next lemma.

\begin{lemma}\label{lem:SeidmanInducednorm}
Let $K$ be a proper cone and let $\norm{\cdot}$ be a cone linear
absolute norm with respect to $K$.
\begin{enumerate} 
\item If $A \geq^K 0$ then 
\begin{equation}\label{eq:norm_f}
\norm{A} = \sup_{x\in K}\frac{\norm{Ax}}{\norm{x}}. 
\end{equation}

\item If $A_i \geq^K B_i \geq^K 0$ for every $i=1$, $\dotsc$, $k$ then
\begin{equation}\label{eq:ProdMono}
\norm{A_k \dotsm A_1}
\geq 
\norm{B_k \dotsm B_1}. 
\end{equation}
\end{enumerate}
\end{lemma}

\begin{proof}
The first statement \changed{can be found
in~{\cite{Seidman2005}}}{follows from
\cite[Theorem~2.1]{Seidman2005}}. The second one is also proved
in~\cite{Seidman2005} when $k=1$. Then the general case follows from
the obvious relationship~$A_k \dotsm A_1 \geq^K B_k \dotsm B_1\geq^K
0$.
\end{proof}

\subsection{Lifts and Kronecker Products}

\added{Another notion that is used extensively in this paper is the lift
of real vectors.} Let $p\geq 1$ be an integer and let $x\in
\mathbb{R}^n$. The $p$-lift (see, e.g., \cite{Parrilo2008}) of~$x$,
denoted by~$x^{[p]}$, is defined as the real vector of length~$n_p =
\binom{n+p-1}{p}$ with its elements being the lexicographically
ordered monomials~$\sqrt{\alpha!}\,x^\alpha$ indexed by all the
possible exponents $\alpha = (\alpha_1, \dotsc, \alpha_n) \in \{0, 1,
\dotsc, p\}^n$ such that \mbox{$\alpha_1 + \cdots + \alpha_n = p$},
where $\alpha! := {p!}/({\alpha_1! \dotsm \alpha_n!})$. For $A\in
\mathbb{R}^{n\times n}$ we define the $n_p\times n_p$ matrix $A^{[p]}$
as the unique matrix~\cite{Jungers2009} satisfying~$(Ax)^{[p]} =
A^{[p]} x^{[p]}$ for every $x\in\mathbb{R}^n$. For a subset~$\mathcal
M$ of $\mathbb{R}^{n\times n}$ we define $\mathcal M^{[p]} = \{
M^{[p]} : M\in \mathcal M\}$. Also for real matrices $A$ and $B$, $A
\otimes B$ denotes the Kronecker product~\cite{Brewer1978} of $A$
and~$B$. Define the Kronecker power $A^{\otimes p}$ by $A^{\otimes 1}
:= A$ and $A^{\otimes (p)} = A^{\otimes (p-1)}\otimes A$ recursively
for a general $p$. We define ${\mathcal M}^{\otimes p} := \{
M^{\otimes p}: M\in \mathcal M \}$. It is known that if $AB$ is
defined then
\begin{equation}\label{eq:PowerKron}
(AB)^{\otimes p} = A^{\otimes p} B^{\otimes p}.
\end{equation}

The next lemma collects some properties of $p$-lifts and Kronecker
products proved in \cite{Blondel2005}.

\begin{lemma}[\cite{Blondel2005}]\label{lem:lifts}
Let $\mathcal M \subset \mathbb{R}^{n\times n}$.
\begin{enumerate}
\item 
$\mathcal M^{[2]}$ leaves a proper cone invariant. 

\item If $\mathcal M$ leaves a proper
cone invariant then ${\mathcal M}^{\otimes p}$ also leaves a proper
cone invariant for every $p\geq 1$.
\end{enumerate}
\end{lemma}

For a probability distribution~$\mu$ on~$\mathbb{R}^{n\times n}$ we
define the probability distribution $\mu^{\otimes p}$ on
$\mathbb{R}^{n^p \times n^p}$ as the
image~\cite[\added{Section~3.6}]{Bogachev2007}\label{page:def:image} of
$\mu$ under the measurable mapping~$(\cdot)^{\otimes p} \colon
\mathbb{R}^{n\times n} \to \mathbb{R}^{n^p\times n^p}$. Let $f$ be a
measurable function on $\mathbb{R}^{n^p\times n^p}$. If $A$ and $B$
are independent random variables following $\mu$ and $\mu^{\otimes
p}$, respectively, then we can show that
\begin{equation}\label{eq:KronExpct}
E[f(B)] = E[f(A^{\otimes p})]. 
\end{equation}
\added{We also define $\mu^{[p]}$ as the image of $\mu$ under
$(\cdot)^{[p]} \colon \mathbb{R}^{n\times n} \to \mathbb{R}^{n_p\times
n_p}$.}

\section{Joint Spectral Characteristics}\label{sec:JSCs}

This section briefly overviews the notions of joint spectral radius
and $L^p$-norm joint spectral radius. The {joint spectral
radius}~\cite{Jungers2009} of a \changed{compact}{bounded}
set~$\mathcal M \subset \mathbb{R}^{n\times n}$ is defined by
\begin{equation*} 
\hat \rho(\mathcal M) 
:=
\adjustlimits \limsup_{k\to\infty} \sup_{A_1, \dotsc, A_k \in \mathcal M} 
\norm{A_k \dotsm A_1}^{1/k}. 
\end{equation*}
One of the important applications of joint spectral radius is in the
stability theory of switched linear systems~\cite{Lin2009}. Define the
switched linear system~$\Sigma_{\mathcal M}$  by
\begin{equation}\label{eq:SigmaM}
\Sigma_{\mathcal M} : x(k+1) = A_k x(k),\ A_k \in \mathcal M
\end{equation}
where $x(0) = x_0 \in \mathbb{R}^n$ is a constant vector.  We say that
$\Sigma_{\mathcal M}$ is {\it absolutely exponentially
stable}~\cite{Gurvits1995} if there exist $C>0$ and $\gamma \in [0,
1)$ such that $\norm{x(k)}\leq C \gamma^k \norm{x_0}$ for every
$\mathcal{M}$-valued sequence~$\{A_k\}_{k=0}^\infty$ and $x_0$. This
stability is characterized by joint spectral radius as follows (see,
e.g., \cite{Jungers2009}).

\begin{proposition}\label{prop:AAS.char}
$\Sigma_{\mathcal M}$ is absolutely exponentially stable if and only if 
$\hat\rho(\mathcal M) < 1$.
\end{proposition}

The following lemma lists some other properties of joint spectral
radius. To state the lemma we recall that the set~$\mathfrak
K(\mathbb{R}^{n\times n})$ of compact and nonempty subsets
of~$\mathbb{R}^{n\times n}$ becomes a complete metric
space~\cite{Wirth2002} if it is endowed with the Hausdorff metric
given by
\begin{equation}\label{eq:Hmetric}
H(\mathcal M, \mathcal N) :=
\max\Bigl\{
\max_{A\in\mathcal M} d(A, \mathcal N), \max_{B\in\mathcal N}d(B, \mathcal{M})
\Bigr\}. 
\end{equation}

\begin{lemma}\label{lem:basic.JSR} The following statements are true.
\begin{enumerate}
\item The restriction of the mapping~$\hat \rho$ to the metric
space~$\mathfrak K(\mathbb{R}^{n\times n})$ 
is continuous~\cite[\added{Lemma~3.5}]{Wirth2002}.

\item It holds \cite[\added{Proposition~2.5}]{Jungers2009} that, for any
$p\geq 1$,
\begin{equation}\label{eq:p-rad[m]}
\hat \rho(\mathcal M^{[p]}) = \hat \rho(\mathcal M)^p.
\end{equation}
\end{enumerate}
\end{lemma}

Then we turn to the $L^p$-norm joint spectral radius of probability
distributions introduced in~\cite{Ogura2012b}. Let $\mu$ be a
probability distribution on~$\mathbb{R}^{n\times n}$ and let \delled{$A$
and}\delwsed$A_k$ ($k= 1, 2, \dotsc$) be random variables independently
following $\mu$. Also let $p$ be a positive integer. The {\it
$L^p$-norm joint spectral radius} ({\it $p$-radius} for short) of
$\mu$ is defined by
\begin{equation} \label{eq:defn:p-rad}
\rho_{p, \mu}
:=
\lim_{k\to\infty} \left(E[\norm{A_k \dotsm A_1}^p ]\right)^{1/pk}. 
\end{equation}
This definition extends the $L^p$-norm joint spectral radius of a set
of finitely many matrices shown in \eqref{eq:classicGJSR}. One can
check that if $\supp \mu$ is \changed{compact}{bounded} then $\rho_{p,
\mu}$ exists and is finite~\cite{Ogura2012b}. Thus, without being
explicitly stated, we assume that probability distributions appearing
in this paper have a bounded support. Though in general the
computation of $p$-radius is a difficult problem~\cite{Jungers2011b},
the next simple formula for $p$-radius is available under certain
assumptions.

\begin{proposition}[{\cite[\added{Theorem~2.5}]{Ogura2012b}}]\label{prop:p-rad} 
Assume that one of the following conditions is true:
\begin{enumerate}[leftmargin=25pt,label={\upshape A$_{\arabic*}$}.,ref=A$_{\arabic*}$]
\item $p$ is even; \label{item:p.even}
\item $\supp \mu$ leaves a proper cone invariant. \label{item:invariant}
\end{enumerate}
Then 
\delxed{\begin{equation} \label{eq:p-radformula}
\stmathed{\rho_{p, \mu} = \rho(E[A^{\otimes p}])^{1/p},}
\end{equation}}
\added{\begin{equation*}
\rho_{p, \mu} = \rho(E[A^{\otimes p}])^{1/p}, 
\end{equation*}}
\added{where $A$ is a random variable following $\mu$.}
\end{proposition}

\changed{The next lemma collects some other properties of
$p$-radius.}{We also quote from \cite{Ogura2012b} other properties of
$p$-radius that will be used in this paper.}

\begin{lemma}\delxed{{\cite{Ogura2012b}}}\label{lem:basic.p-rad}
Let $p\geq 1$ be arbitrary. 
\begin{enumerate}
\item If $p \leq q$ then $\rho_{p, \mu} \leq \rho_{q, \mu}$.

\item \changed{For any probability distribution~$\mu$,}{It holds that}
\begin{equation}\label{eq:p-rad.geq}
\rho_{p, \mu} \geq \rho(E[A^{\otimes p}])^{1/p}. 
\end{equation}

\item For every $m\geq 1$, 
\begin{equation} \label{eq:p-rad.lift}
\rho_{p, \mu^{\otimes m}} = \rho_{p, \mu^{[m]}} = \rho_{mp, \mu}^m.
\end{equation}

\end{enumerate}
\end{lemma}

\begin{proof}
The first two statements can be found in~\cite{Ogura2012b}. The last
statement can be proved in the same way as \eqref{eq:p-rad[m]}.
\end{proof}

\begin{remark}\label{rmk:normequiv}
By the equivalence of the norms on a finite dimensional vector space,
the value of $L^p$-norm joint spectral radius is independent of the
norm used in~\eqref{eq:defn:p-rad}.
\end{remark}

\section{Limit Formula for Joint Spectral Radius\label{sec:JSRchar}}

This section presents a novel limit formula for joint spectral radius.
We state the next assumption on a probability distribution~$\mu$
on~$\mathbb{R}^{n\times n}$.
\begin{enumerate}[leftmargin=25pt,label={\upshape A$_{\arabic*}$}.,ref=A$_{\arabic*}$]\setcounter{enumi}{2}
\item \label{item:mu.nice} The singular part~$\mu_s$ of $\mu$ is a
linear combination of finitely many Dirac measures, i.e., either
$\mu_s = 0$ or there exist positive numbers~$p_1$, $\dotsc$, $p_N$ and
matrices~$M_1$, $\dotsc$, $M_N$ such that
\begin{equation}\label{eq:mu_s}
\mu_s =  p_1 \delta_{M_1} + \cdots + p_N \delta_{M_N}. 
\end{equation}
\end{enumerate}
Notice that any of the assumptions from \ref{item:p.even} to
\ref{item:mu.nice} does not require \changed{$\mu$ to have a finite
support}{$\supp \mu$ to consist of only finitely many
matrices}. The next theorem is the main result of
this paper\delled{, which extends the characterization of the joint
spectral radius of finite sets of matrices~{\cite{Blondel2005}} to
general compact sets}.
\begin{theorem}\label{thm:main:equiv}
\added{Let $\mu$ be a probability distribution satisfying
\ref{item:invariant} and \ref{item:mu.nice} and let $\mathcal M =
\supp \mu$. Then
\begin{equation*} 
\hat \rho(\mathcal M)
=
\lim_{p\to\infty} \rho_{p, \mu}. 
\end{equation*}}
\end{theorem}

\added{Proposition~\ref{prop:p-rad} allows us to state the theorem in
the following equivalent form, which extends the limit formula of the
joint spectral radius of a set of finitely many matrices given in
\cite{Blondel2005}.}

\begin{theorem}\delxed{\label{thm:Limit}}
Let $\mu$ be a probability distribution satisfying
\ref{item:invariant} and \ref{item:mu.nice} and let $\mathcal M =
\supp \mu$. Then
\begin{equation*} 
\hat \rho(\mathcal M)
=
\lim_{p\to\infty}\rho\left( E[A^{\otimes p}] \right)^{1/p}. 
\end{equation*}
\end{theorem}

\added{If $\mu$ is the uniform distribution on a finite set then the
theorem recovers \cite[Theorem~3]{Blondel2005}.} As a simple
illustration of the present theorem let us see the next example.

\begin{example}\label{ex:limit}
Let $\gamma > 0$ and let $\mu$ be the uniform distribution on
$[0,\gamma]$. Clearly $\mu$ is absolutely continuous and $\mathcal M =
\supp \mu = [0, \gamma]$ leaves the proper cone~\changed{$\{ r: r\geq 0
\}$}{$\mathbb{R}_+$} of~$\mathbb{R}$ invariant. It is easy to observe
$\hat{\rho}(\mathcal M) = \gamma$ and $\rho(E[A^{\otimes p}]) =
\gamma^p/(p+1)$. Therefore $\lim_{p\to\infty} \rho(E[A^{\otimes
p}])^{\added{1/p}} = \gamma = \hat \rho(\mathcal M)$, as expected. The
characterization in~\cite{Blondel2005} cannot be applied to this
simple example as $\mu$ has an infinite support.
\end{example}

\added{We can use the above limit formula to generalize another limit
formula of joint spectral radius given in \cite{Xu2011}.}

\begin{corollary}
If $\mu$ is of the form~\eqref{eq:mu_s} then
\begin{align}
\hat \rho (\mathcal M) 
&=
\lim_{p\to\infty} \rho( E[A^{\otimes (2p)}] )^{1/(2p)} \label{eq:Limit.2p}
\\
&=
\limsup_{p\to\infty} \rho( E[A^{\otimes p}] )^{1/p}. \label{eq:LimitSup}
\end{align}
\end{corollary}

\added{It is remarked that, setting $\mu$ to be the uniform distribution
in this corollary, we can recover Theorem~2.1 in \cite{Xu2011}.}


\begin{proof}
\added{We shall apply Theorem~\ref{thm:main:equiv} to $\mu^{[2]}$, which
satisfies both \ref{item:invariant} and \ref{item:mu.nice} because its
support $\{A_1^{[2]}, \dotsc, A_N^{[2]}\}$ leaves a proper cone
invariant by Lemma~\ref{lem:lifts} and also ${\mu^{[2]}} =
\sum_{i=1}^N p_i \delta_{A_i^{[2]}}$ is a finite sum of point masses.
Therefore Theorem~\ref{thm:main:equiv} shows $\hat \rho(\supp
\mu^{[2]})  =\lim_{p\to\infty} \rho_{p, \mu^{[2]}}$. This equation
implies the first equality~\eqref{eq:Limit.2p} because
\eqref{eq:p-rad[m]} shows $\hat \rho(\supp \mu^{[2]}) = \hat
\rho(\mathcal M)^2$ and also \eqref{eq:p-rad.lift} and
Proposition~\ref{prop:p-rad} yield $\rho_{p, \mu^{[2]}} = \rho_{2p,
\mu}^2 = \rho(E[A^{\otimes (2p)}])^{1/p}$. Then let us show the second
equation~\eqref{eq:LimitSup}. Using the
inequality~\eqref{eq:p-rad.geq}, the monotonicity of $p$-radius
(Lemma~\ref{lem:basic.p-rad}), and Proposition~\ref{prop:p-rad}, we
can show
\begin{equation*}
\begin{aligned}
\rho(E[A^{\otimes (2p-1)}])^{1/(2p-1)}
&\leq
\rho_{2p-1, \mu} 
\\
&\leq
\rho_{2p, \mu} 
\\
&=
\rho(E[A^{\otimes (2p)}])^{1/(2p)}. 
\end{aligned}
\end{equation*}
This inequality and \eqref{eq:Limit.2p} prove the
equation~\eqref{eq:LimitSup}.}
\end{proof}

\subsection{Proof}

\delled{In the rest of this section we give the proof of
Theorem~{\ref{thm:Limit}}.} \added{This section gives the proof of
Theorem~\ref{thm:main:equiv}. Let $\mu$ be a probability distribution
on $\mathbb{R}^{n\times n}$ with bounded support. First we observe
that the definitions of $p$-radius and joint spectral radius
immediately show $\hat \rho(\mathcal M) \geq \rho_{p,\mu}$. Since
$\rho_{p, \mu}$ is non-decreasing with respect to $p$ by
Lemma~\ref{lem:basic.p-rad}, the limit $\lim_{p\to\infty} \rho_{p,
\mu}$ exists and satisfies
\begin{equation}\label{eq:main:pre}
\hat \rho(\mathcal M) \geq 
\lim_{p\to\infty} 
\rho_{p, \mu}.
\end{equation}}
Therefore, to prove Theorem~\ref{thm:main:equiv}, we need to
show\delws\del{the next theorem.}
\add{\begin{equation} \label{eq:main:oppst}
\hat \rho(\mathcal M) \leq 
\lim_{p\to\infty} 
\rho_{p, \mu}
\end{equation}
under the assumption that $\mu$ satisfies \ref{item:invariant}
and \ref{item:mu.nice}}
\delx{\begin{theorem}\label{thm:oppst}
{\del{Let $\mu$ be a probability distribution satisfying
{\ref{item:invariant}} and {\ref{item:mu.nice}} and let $\mathcal M =
\supp \mu$. Then}
\begin{equation} \label{eq:main:oppstold}
\stmath{\hat \rho(\mathcal M) \leq 
\lim_{p\to\infty} 
\rho_{p, \mu}.}
\end{equation}}
\end{theorem}}
In the rest of this section we prove
\change{Theorem~{\ref{thm:oppst}}}{inequality~\eqref{eq:main:oppst}}.
In the sequel it is assumed that $\mu$ satisfies both
\ref{item:invariant} and \ref{item:mu.nice}. Let $\mathcal M = \supp
\mu$ and let $K$ be the proper cone left invariant by $\mathcal M$. 

\added{We first note that showing \eqref{eq:main:oppst} is not as
straightforward as showing \eqref{eq:main:pre}. Inequality
\eqref{eq:main:oppst} means that the maximum growth rate~$\hat
\rho(\mathcal M)$ of the products of matrices from $\mathcal M$ can be
attained by the $p$th averaged growth rate~$\rho_{p, \mu}$ of the
products as $p\to\infty$. The difficulty in showing the inequality is
that the set of sequences giving the maximum growth rate can have a
very small probability, possibly zero, and therefore we should not
expect that the rate can be captured by the averaged growth rate. For
example, the joint spectral radius $\gamma$ in Example~\ref{ex:limit}
results from the singleton~$\{\gamma\}$, which is a null set for
$\mu$.}

\added{We avoid the above mentioned problem by first focusing on
well-behaving subsets of $\mathcal M$, and then approximating
$\mathcal M$ by a sequence of such subsets. For $M\in
\mathbb{R}^{n\times n}$ let
\begin{equation*}
\mathcal S_M := \{A\in\mathbb{R}^{n\times n} : A \geq^K M \}.
\end{equation*}
Then we define $\mathfrak M$\label{pg:def:frakM} as a family of
measurable and nonempty subsets $\mathcal N$ of $\mathcal M$
satisfying the following property: for every $\epsilon>0$ there exists
$\delta>0$ such that
\begin{equation}\label{eq:muS}
\mu(\mathcal S_{(1-\epsilon)M})\geq \delta
\end{equation}
for every $M\in \mathcal N$. The next proposition shows that the joint
spectral radius of a subset belonging to $\mathfrak M$ admits an
estimate of the form \eqref{eq:main:oppst}. Roughly speaking,
inequality~\eqref{eq:muS} will be used to guarantee that $\mu$ is
always ``aware'' of products of matrices with almost maximum growth
rates. The uniformity of the lower bound~$\delta$ with respect to $M$
plays a key role.}

\begin{proposition}\label{prop:frakM}
\added{If $\mathcal N \in \mathfrak M$ then $\hat \rho(\mathcal N)
\leq \lim_{p\to\infty} \rho_{p, \mu}$.}
\end{proposition}

\begin{proof}
\added{If $\hat \rho(\mathcal N) = 0$ then the inequality holds
vacuously. Assume $\hat \rho(\mathcal N) > 0$. Then, without loss of
generality we can assume $\hat \rho(\mathcal N) = 1$ by scaling
matrices in $\mathcal N$ by the factor~$1/\hat \rho(\mathcal N)$. Take
a cone linear absolute norm~$\norm{\cdot}$ with respect to~$K$. By
Proposition~\ref{prop:AAS.char}, there exist $c>0$ and
$\{M_k\}_{k=1}^\infty \subset \mathcal N$ such that $\norm{M_k \dotsm
M_1} > c$ for infinitely many $k$. Take an arbitrary $\gamma < 1$ and
define $\epsilon := 1-\gamma$. Let us take the corresponding
$\delta>0$ satisfying \eqref{eq:muS}. Observe that if $A_i \in
\mathcal S_{(1-\epsilon)M_i}$ then $A_i \geq^K (1-\epsilon)M_i =
\gamma M_i$ so that, by \eqref{eq:ProdMono}, we have  $\norm{A_k
\dotsm A_1} \geq \gamma^k \norm{M_k\dotsm M_1} > c \gamma^k$.
Therefore
\begin{equation*}
\mu^k\left(\{(A_1, \dotsc, A_k) : \norm{A_k\dotsm A_1} 
> 
c\gamma^k\}\right)  
\geq 
\prod_{i=1}^k \mu( \mathcal S_{(1-\epsilon)M_i} ) 
\geq \delta^k
\end{equation*}
and hence, by Markov's inequality, we obtain $E[\norm{A_k\dotsm
A_1}^p] > \delta^k (c \gamma^{k})^p$, which implies $E[\norm{A_k\dotsm
A_1}^p]^{1/kp} > \delta^{1/p} c^{1/k}  \gamma$. Taking the limit
$k\to\infty$ in this inequality shows $\rho_{p, \mu} \geq
\delta^{1/p}\gamma$ by Remark~\ref{rmk:normequiv}. Thus we obtain
$\lim_{p\to\infty}\rho_{p,\mu}\geq \gamma$. Since $\gamma$ can be made
arbitrarily close to $1$ we see $\lim_{p\to\infty}\rho_{p,\mu}\geq 1 =
\hat \rho(\mathcal N)$, as desired.}
\end{proof}

\added{If we could show that $\mathcal M$ is in $\mathfrak M$ then
Proposition~\ref{prop:frakM} proves \change{Theorem~{\ref{thm:oppst}}}{inequality \eqref{eq:main:oppst}}. In this
paper, however, we leave open the problem of checking $\mathcal M \in
\mathfrak M$ and take another approach via the approximation of
$\mathcal M$ by elements in $\mathfrak M$.} \changed{The proof relies
on two propositions, each of which is preceded by a lemma. When $\mu$
satisfies {\ref{item:mu.nice}} we can uniquely decompose}{Let us
decompose} $\mu$ as
\begin{equation}\label{eq:decomp}
\mu = \mu_c + \mu_s, 
\end{equation}
where $\mu_c$ is an absolutely continuous measure and $\mu_s$ is
either the zero measure or is of the
form~\eqref{eq:mu_s}.\delwsed\delled{We write $\mathcal M:= \supp \mu$,
$\mathcal{M}_c := \supp \mu_c$ and $\mathcal{M}_s := \supp \mu_s$.}
Clearly \changed{$\mathcal M = \mathcal M_c \cup \mathcal
M_s$}{$\mathcal M = (\supp \mu_c) \cup (\supp \mu_s)$}. \changed{For a
proper cone~$K$ of $\mathbb{R}^n$ and $r\geq 0$}{For $r>0$} we define
\begin{equation*}
\pi_r := \left\{
M\in\pi : d(M, \partial \pi )\geq r
\right\}.
\end{equation*}
Notice that $\pi_r \subset \interior \pi$\delled{ if $r>0$}. Finally we
let
\delxed{\begin{equation*}
\stmathed{\mathcal M_r := (\mathcal M_c \cap \pi_r) \cup \mathcal M_s.}
\end{equation*}}
\added{\begin{equation*}
\mathcal M_r := (\pi_r \cap \supp \mu_c) \cup \supp \mu_s.
\end{equation*}}
\added{We shall show that this $\mathcal M_r$ is indeed in $\mathfrak M$
and, furthermore, as $r\to 0$, $\hat \rho(\mathcal M_r)$ converges to
$\hat \rho(\mathcal M)$.  Let us begin with the next observation.}

\begin{lemma}
\added{There exists $r_0 > 0$ such that $\mathcal M_r \subset \mathfrak
K(\mathbb{R}^{n\times n})$ for every $r < r_0$.}
\end{lemma}

\begin{proof}
\added{Since $\mathcal M_r$ is always compact, we need  show that
$\mathcal M_r$ is not empty for every $r< r_0$ for some $r_0>0$. Since
$\mathcal M_r$ is decreasing with respect to $r$, it is sufficient to
show that there exists $r_0>0$ such that $\mathcal M_{r_0}$ is
nonempty. Assume the contrary, i.e., $\mathcal{M}_r = \emptyset$ for
every $r>0$. Then it must be that $\supp \mu_c \subset \partial\pi$
and $\mu_s = 0$. The latter condition shows that $\mu_c$ is nonzero.
Thus the former condition shows that the nonzero and absolutely
continuous measure $\mu_c$ is concentrated on a null set by
Lemma~\ref{lem:nullset}, which is a contradiction.}
\end{proof}

\added{In the sequel we assume $r< r_0$. In order to show $\mathcal M_r
\in \mathfrak M$ we will need the next lemma.}

\begin{lemma}\label{lem:SMSM.new}
\added{ {Assume that $\mu$ is absolutely continuous (i.e., $\mu_s = 0$).
Let $M\in\mathbb{R}^{n\times n}$ be arbitrary.} If a sequence~$\{M_k
\}_{k=1}^\infty \subset \mathbb{R}^{n\times n}$ converges to~$M$ then
\begin{align}
\lim_{k\to\infty} \mu(\mathcal S_{M}\backslash \mathcal S_{M_k}) &= 0, 
\label{eq:lim.first.new}
\\
\lim_{k\to\infty} \mu(\mathcal S_{M_k}\backslash \mathcal S_M) &= 0. 
\notag
\end{align}}
\end{lemma}

\begin{proof}
\added{We only prove the first equation~\eqref{eq:lim.first.new} because
the second one can be proved in a similar way. By shifting the
point~$M$ to the origin and also $\mu$ accordingly, without loss of
generality we can assume that $M=0$. In this case $\mathcal S_M =
\pi$. Notice that the boundary $\partial \pi$ is a null set with
respect to $\mu$ by Lemma~\ref{lem:nullset}. Therefore it is
sufficient to show that, for every $A\in \mathcal M \backslash
\partial \pi \subset \interior \pi$,
\begin{equation}\label{eq:SMSM=0.new}
\lim_{k\to\infty} \chi_{\pi \backslash \mathcal S_{M_k}}(A) = 0, 
\end{equation}
where $\chi_S$ denotes the characteristic function for a set~$S$.
Since $A \in \interior \pi$ we have $A>^K 0$ by \eqref{eq:intpi(K)}.
Therefore the origin is in the open set~$G = \{B \in
\mathbb{R}^{n\times n}: B <^K A\} \added{ = A - \interior \pi}$. Hence,
since $\{M_k\}_{k=1}^\infty$ converges to $0$, if $k$ is sufficiently
large then $M_k$ is in $G$, i.e.,  $A >^K M_k$, which shows $A\in
\mathcal S_{M_k}$. Thus \eqref{eq:SMSM=0.new} actually holds.}
\end{proof}

\added{Then we can prove the next proposition.}

\begin{proposition} 
\added{If $0< r < r_0$ then $\mathcal M_r \in \mathfrak M$.}
\end{proposition}

\begin{proof}
\added{Fix $\epsilon > 0$ and $r>0$.  Define $\phi \colon \mathcal M_r
\to \mathbb{R}$ by $\phi(M)  := \mu(\mathcal S_{(1-\epsilon)M})$.
First assume that $\mu$ is absolutely continuous with respect to the
Lebesgue measure, i.e., $\mu_s = 0$. Since $\mathcal M_r$ is compact,
it is sufficient to show that $\phi$ is continuous and positive. In
order to show that $\phi$ is continuous at~$M \in \mathcal M_r$, let
$\{M_k \}_{k=1}^\infty$ be a sequence of $\mathcal M_r$ converging to
$M$. Then we can see that
\begin{equation*}
\abs{\phi(M) -  \phi(M_k)}
\leq
\mu(\mathcal S_{(1-\epsilon)M} \backslash \mathcal S_{(1-\epsilon)M_k} ) 
+ 
\mu(\mathcal S_{(1-\epsilon)M_k} \backslash \mathcal S_{(1-\epsilon)M} ), 
\end{equation*}
which converges to $0$ as $k\to\infty$ by Lemma~\ref{lem:SMSM.new}.
Therefore $\phi$ is continuous. Then let us show that $\phi(M)>0$.
Since $\mathcal M_r \subset \interior \pi$ we have $M>^K 0$ and
therefore $M >^K (1-\epsilon)M$. This shows that the open
set~$\interior \mathcal S_{(1-\epsilon)M}$ and $\supp \mu$ has a
nonempty intersection containing $M$. Therefore $\mu(\interior
\mathcal S_{(1-\epsilon)M}) > 0$ and hence $\phi(M) \geq \mu(\interior
\mathcal S_{(1-\epsilon)M}) >0$, as desired.}

\added{Then we consider the general case. Decompose $\mu$ as
\eqref{eq:decomp}. On the one hand, from the above argument, there
exists a constant~$\delta_c > 0$ such that $\mu_c(\mathcal
S_{(1-\epsilon)M}) \geq \delta_c$ for every $M \in \pi_r \cap \supp
\mu_c$. On the other hand, if $M \in \supp \mu_s$ then $M = M_i \geq^K
0$ for some $1\leq i\leq N$. Therefore $M_i \in \mathcal
S_{(1-\epsilon)M}$ because $M_i \geq^K (1-\epsilon)M$. Hence $\phi(M)
\geq \mu(\{M_i \}) =  p_i > 0$. Thus we can see that $\delta =
\min(\delta_c, p_1, \dotsc, p_N)$ is a desired constant. This argument
is valid even when $\mu_c = 0$ and therefore completes the proof.}
\end{proof}

\added{This proposition shows $\hat \rho(\mathcal M_r) \leq
\lim_{p\to\infty} \rho_{p, \mu}$ for every $0<r<r_0$ by
Proposition~\ref{prop:frakM}. Therefore, if we could show
\begin{equation}\label{eq:hatrholim}
\hat \rho(\mathcal M) = \lim_{r\to 0} \hat \rho(\mathcal M_r)
\end{equation}
then we can complete the proof of
\change{Theorem~{\ref{thm:oppst}}}{inequality \eqref{eq:main:oppst}}.
The rest of this section is devoted to the proof of
\eqref{eq:hatrholim}. For the proof we will need the next lemma.}

\begin{lemma}
\changed{Let $K$ be a proper cone. If $\mu$ satisfies
{\ref{item:mu.nice}} then, for}{For} every $A\in \mathcal M$
 \added{it holds that}
\begin{equation}\label{eq:SMSM}
\lim_{r \to 0}d(A, \mathcal M_r) = 0.
\end{equation}
\end{lemma}

\begin{proof}
Let $A \in \mathcal M$ be arbitrary.\delwsed\delled{Notice that the set
$\mathcal M_r$ is decreasing with respect to $r$ so that the
function~$d(A, \mathcal M_r)$ of~$r$ is increasing. Therefore the
limit in {\eqref{eq:SMSM}} is well defined.} First we suppose $\mu_s =
0$. Let $\epsilon>0$ be arbitrary and let $\mathcal B$ denote the open
ball in $\mathbb{R}^{n\times n}$ with center~$A$ and
radius~$\epsilon$. Since $\mu(\mathcal B) > 0$ and the set
\changed{$\{A\} \cup \partial \pi$}{$\partial \pi$} has the Lebesgue
measure~$0$ by Lemma~\ref{lem:nullset}, $\mathcal B \cap \mathcal M$
is not contained in \changed{$\{A\} \cup \partial \pi$}{$\partial
\pi$}. Therefore we can take $M \in \mathcal B \cap \mathcal M$ that
\delled{is distinct from $A$ and}{\delwsed}is not in $\partial \pi$.
\added{Let $\delta:= d(M, \partial \pi) > 0$ and assume $r\leq \delta$.
Then $M \in \mathcal M_\delta \subset \mathcal M_r$ and therefore}
\delled{Now, if $r < d(B, \partial \pi)$ then we have}{\delwsed}$d(A,
\mathcal M_r) \leq \norm{A-M} < \epsilon$\delled{ because $M \in \mathcal
M_r$}. This shows \eqref{eq:SMSM} since $\epsilon>0$ was arbitrary.

Then let us consider the general case. If $A \in \supp \mu_c$ then,
since  $\mathcal M_r \supset \pi_r \cap \supp \mu_c$ we can see
$\lim_{r\to 0}d(A, \mathcal M_r) \leq \lim_{r\to 0} d(A, \pi_r \cap
\supp \mu_c ) = 0$, where in the last equation we applied the above
argument for the special case $\mu_s = 0$ to the normalized
probability measure $\mu_c/(\mu_c(\mathbb{R}^{n\times n}))$. On the
other hand, if $A \in \supp \mu_s$ then \eqref{eq:SMSM} clearly holds
because $A$ is also in $\mathcal{M}_r$ and hence $d(A, \mathcal{M}_r)
= 0$ for every $r$. This completes the proof.
\end{proof}

\changed{Using this lemma we can prove the following proposition on the
Hausdorff distance~{\eqref{eq:Hmetric}} between $\mathcal M$ and
$\mathcal M_r$}{Now we can prove \eqref{eq:hatrholim}}.

\delxed{\begin{proposition} \label{prop:limHausdMetric}
\delled{If $\mu$ satisfies {\ref{item:mu.nice}} then}
\begin{equation} \label{eq:limHausdMet}
\stmathed{\lim_{r\to 0} H(\mathcal M,  \mathcal M_r) = 0.}
\end{equation}
\end{proposition}}

\begin{proof}[\added{Proof of \eqref{eq:hatrholim}}]
\added{By Lemma~\ref{lem:basic.JSR} it is sufficient to show $\lim_{r\to
0} H(\mathcal M,  \mathcal M_r) = 0$.} \changed{Let us first see that
the limit in {\eqref{eq:limHausdMet}} is well defined. Clearly
$\mathcal M$ is in~$\mathfrak K(\mathbb{R}^{n\times n})$ because $\mu$
is assumed to have a bounded support. Then let us show that, if $r>0$
is sufficiently small, the set $\mathcal{M}_r$ also is in $\mathfrak
K(\mathbb{R}^{n\times n})$. Assume the contrary, i.e., $\mathcal{M}_r
= \emptyset$ for every $r>0$. Then it must be that $\mathcal M_c
\subset \partial\pi$ and $\mu_s = 0$. The latter condition shows that
$\mu_c$ is nonzero. Thus the former condition shows that the nonzero
and absolutely continuous measure $\mu_c$ is concentrated on the null
set $\partial\pi$, which is a contradiction. Therefore the distance
$H(\mathcal M,  \mathcal M_r)$ is well defined at least for a
sufficiently small $r$.}{Notice that the distance $H(\mathcal M,
\mathcal M_r)$ is well defined for each $r$.} \changed{Finally,
since}{Since} the set~$\mathcal M_r$ is
\changed{increasing}{decreasing} with respect to~$r$, the
distance~$H(\mathcal M, \mathcal M_r)$ is
\changed{decreasing}{increasing} with respect to~$r$ so that the
limit~$\lim_{r\to 0} H(\mathcal M, \mathcal M_r)$ does exist.
\changed{Now}{We} assume $\lim_{r\to 0} H(\mathcal M,  \mathcal M_r) >
0$ to derive a contradiction. In this case there exists $\epsilon>0$
such that $H(\mathcal M, \mathcal M_r) > \epsilon$ for every $r > 0$.
By the definition of the Hausdorff metric~\eqref{eq:Hmetric} this
implies $\max_{A\in \mathcal M} d(A, \mathcal M_r) > \epsilon$ because
$\mathcal M_r \subset \mathcal M$. Therefore there exists $A_r \in
\mathcal M$ such that $d(A_r, \mathcal M_r) > \epsilon$ for each $r >
0$. Now let $\{r_i\}_{i=1}^\infty$ be a positive sequence decreasingly
converging to $0$. Since $\mathcal M$ is compact, there exists a
subsequence~$\{r'_i \}_{i=1}^\infty$ of~$\{r_i  \}_{i=1}^\infty$ such
that $\{A_{r'_i} \}_{i=1}^\infty$ converges to some $A \in \mathcal
M$. Using the triangle inequality we can show
\begin{equation*}
\begin{aligned}
d(A, \mathcal M_{r'_i})
& \geq 
d(A_{r'_i}, \mathcal M_{r'_i}) - d(A_{r'_i}, A) 
\\
&>
\epsilon - d(A_{r'_i}, A)
\end{aligned}
\end{equation*}
and hence $\limsup_{r\to0} d(A, \mathcal M_r) \geq \epsilon$, which
contradicts to \eqref{eq:SMSM}.
\end{proof}

%

\delled{To prove another proposition we need the next lemma. }

\delxed{\begin{lemma}\label{lem:SMSM}
\delled{Let $K$ be a proper cone in $\mathbb{R}^n$. For $M\in
\mathbb{R}^{n\times n}$ and define}
\begin{equation*}
\stmathed{\mathcal S_M := \{A\in\mathbb{R}^{n\times n} : A \geq^K M \}.}
\end{equation*}
\delled{Let $\mu$ be a probability distribution on~$\mathbb{R}^{n\times n}$
that is absolutely continuous with respect to the Lebesgue measure. If
a sequence~$\{M_k \}_{k=1}^\infty \subset \mathbb{R}^{n\times n}$
converges to~$M$ then}
\begin{align}
\stmathed{\lim_{k\to\infty} \mu(\mathcal S_{M}\backslash \mathcal S_{M_k}) }
&
\stmathed{= 0,} 
\label{eq:lim.first}
\\
\stmathed{\lim_{k\to\infty} \mu(\mathcal S_{M_k}\backslash \mathcal S_M) }
&
\stmathed{= 0.} 
\notag
\end{align}
\end{lemma}}

\delxed{\begin{proof}
\delled{By shifting the point~$M$ to the origin and also $\mu$
accordingly, without loss of generality we can assume that $M=0$. In
this case $\mathcal S_M = \pi$. We only prove the first
equation~{\eqref{eq:lim.first}} because the second one can be proved
in a similar way. Notice that the boundary $\partial \pi$ is a null
set with respect to $\mu$ by Lemma~{\ref{lem:nullset.old}}. Therefore
it is sufficient to show that, for every $A\in \mathcal M \backslash
\partial \pi$,}
\begin{equation}\label{eq:SMSM=0}
\stmathed{\lim_{k\to\infty} \chi_{\pi \backslash \mathcal S_{M_k}}(A) = 0,} 
\end{equation}
\delled{where $\chi_S$ denotes the characteristic function for a set~$S$.
Since $\pi$ is closed, $A$ is either not in~$\pi$ or in~$\interior
\pi$. In the former case {\eqref{eq:SMSM=0}} clearly holds. Next
assume $A \in \interior \pi$. Then $A>^K 0$ by {\eqref{eq:intpi(K)}}
so that the origin $0$ is in the open set~$G := \{B \in
\mathbb{R}^{n\times n}: B <^K A\}$. Therefore, since
$\{M_k\}_{k=1}^\infty$ converges to $0$, if $k$ is sufficiently large
then $M_k$ is in $G$, i.e.,  $A >^K M_k$ and therefore $A\in \mathcal
S_{M_k}$. Thus {\eqref{eq:SMSM=0}} holds.}
\end{proof}}

\delled{Then we can prove the next proposition.}

\delxed{\begin{proposition}\label{prop:LowBound.old}
\delled{Assume that $\mu$ satisfies {\ref{item:invariant}} and
{\ref{item:mu.nice}}. Let $K$ be a proper cone left invariant by
$\mathcal M$. Then, for every $\epsilon>0$ and $r>0$, there exists
$\delta>0$ such that $\mu(\mathcal S_{(1-\epsilon)M})\geq \delta$ for
every $M\in \mathcal M_r$.}
\end{proposition}}

\delxed{\begin{proof}
\delled{Fix $\epsilon > 0$ and $r>0$.  Define $\phi \colon \mathcal M_r
\to \mathbb{R}$ by $\phi(M)  := \mu(\mathcal S_{(1-\epsilon)M})$.
First assume that $\mu$ is absolutely continuous with respect to the
Lebesgue measure, i.e., $\mu_s = 0$. Since $\mathcal M_r$ is compact,
to prove {\eqref{eq:muS}}, it is sufficient to show that $\phi$ is
continuous and positive. In order to show that $\phi$ is continuous
at~$M \in \mathcal M_r$, let $\{M_k \}_{k=1}^\infty$ be a sequence of
$\mathcal M_r$ converging to $M$. Then we can see that}
\begin{equation*}
\stmathed{\abs{\phi(M) -  \phi(M_k)}
\leq
\mu(\mathcal S_{(1-\epsilon)M} \backslash \mathcal S_{(1-\epsilon)M_k} ) 
+ 
\mu(\mathcal S_{(1-\epsilon)M_k} \backslash \mathcal S_{(1-\epsilon)M} ).}
\end{equation*}
\delled{Therefore $\phi$ is continuous by Lemma~{\ref{lem:SMSM}}. Then
let us show that $\phi$ is positive at~$M$. Since $\mathcal M_r
\subset \interior \pi$ we have $M>^K 0$ and therefore $M >^K
(1-\epsilon)M$. This gives $(\interior \mathcal S_{(1-\epsilon)M})
\cap \supp \mu \ni M$ and thus $\phi(M) \geq \mu(\interior \mathcal
S_{(1-\epsilon)M}) >0$.}
\\
\delled{Then we consider the general case. Decompose $\mu$ as
{\eqref{eq:decomp}}. On the one hand, from the above argument, there
exists a constant~$\delta_c > 0$ such that $\mu_c(\mathcal
S_{(1-\epsilon)M}) \geq \delta_c$ for every $M \in \mathcal M_c \cap
\pi_r$. On the other hand, if $M \in \mathcal M_s$ then $M = M_i
\geq^K 0$ for some $1\leq i\leq N$. Therefore $M_i \in \mathcal
S_{(1-\epsilon)M}$ because $M_i \geq^K (1-\epsilon)M$. Hence $\phi(M)
\geq \mu(\{M_i \}) =  p_i > 0$. Thus we can see that $\delta :=
\min(\delta_c, p_1, \dotsc, p_N) > 0$ is a desired constant. This
argument is valid even when $\mu_c = 0$ and therefore completes the
proof.}
\end{proof}}

\delled{Now we are at the position to prove the main result of
Theorem~{\ref{thm:Limit}}.}

\delxed{\begin{proof}[Proof of Theorem~\ref{thm:Limit}]
\delled{It is sufficient to show }\delled{$\hat \rho(\mathcal M) =
\lim_{p\to\infty}\rho_{p, \mu}$}\delled{ by {\ref{item:invariant}} and
{\eqref{eq:p-radformula}}. First notice that the definitions of
$p$-radius and joint spectral radius immediately show $\rho_{p,\mu} 
\leq  \hat \rho(\mathcal M)$. Therefore $\lim_{p\to\infty}\rho_{p,
\mu} \leq \hat \rho (\mathcal M)$ because $\rho_{p, \mu}$ is
non-decreasing with respect to $p$ by the first statement of
Lemma~{\ref{lem:basic.p-rad}}.}
\\
\delled{To show the opposite direction let $r>0$ be arbitrary and let
$\gamma_r := \hat \rho (\mathcal M_r)$. Let $K$ be a proper cone left
invariant by $\mathcal M$ and take any cone linear absolute
norm~$\norm{\cdot}$ with respect to~$K$. By
Proposition~{\ref{prop:AAS.char}}, there exist $C>0$ and
$\{M_k\}_{k=1}^\infty \subset \mathcal M_r$ such that}
\begin{equation}\label{eq:nealy}
\stmathed{\norm{M_k \dotsm M_1} > C\gamma_r^k}
\end{equation}
\delled{for infinitely many $k$. Take an arbitrary $\beta_r$ such that
$\gamma_r>\beta_r$ and define \mbox{$\epsilon :=
(\gamma_r-\beta_r)/\gamma_r$}. Let us take the corresponding
$\delta>0$ given by Proposition~{\ref{prop:LowBound.old}}. Observe that if
$A_i \in \mathcal S_{(1-\epsilon)M_i}$ then $A_i \geq^K (1-\epsilon)M_i =
(\beta_r/\gamma_r)M_i$ so that, by {\eqref{eq:ProdMono}} and
{\eqref{eq:nealy}}, we have  $\norm{A_k \dotsm A_1} \geq
(\beta_r/\gamma_r)^k  \norm{M_k\dotsm M_1} > C \beta_r^k$. Therefore}
\begin{equation*}
\begin{aligned}
\stmathed{E[\norm{A_k\dotsm A_1}^p]
> }&
\stmathed{ \left(C \beta_r^k\right)^p \mu^k\left(\{(A_1, \dotsc, A_k) : \norm{A_k\dotsm A_1} > C\beta_r^k\}\right) }
\\	
\stmathed{\geq} &
\stmathed{C^p \beta_r^{pk} \prod_{i=1}^k \mu( \mathcal S_{(1-\epsilon)M_i} ) }
\\
\stmathed{\geq}&
\stmathed{ C^p \beta_r^{pk} \delta^k}
\end{aligned}
\end{equation*}
\delled{and hence $E[\norm{A_k\dotsm A_1}^p]^{1/kp} > C^{1/k}
\delta^{1/p} \beta_r$ for infinitely many $k$. Taking the limit
$k\to\infty$ in this inequality shows $\rho_{p, \mu} \geq
\delta^{1/p}\beta_r$ by Remark~{\ref{rmk:normequiv}}. Thus we obtain
$\lim_{p\to\infty}\rho_{p,\mu}\geq \beta_r$. Since $\beta_r$ can be
made arbitrarily close to $\gamma_r$ we see
$\lim_{p\to\infty}\rho_{p,\mu}\geq \gamma_r$, which further shows
$\lim_{p\to\infty}\rho_{p,\mu}\geq \hat \rho (\mathcal M)$ since
$\gamma_r = \hat \rho(\mathcal M_r) \to \hat{\rho}(\mathcal M)$ as
$r\to 0$ by Lemma~{\ref{lem:basic.JSR}} and
Proposition~{\ref{prop:limHausdMetric}}. This completes the proof.}
\end{proof}}

\delxed{\delled{Finally the next corollary of Theorem~{\ref{thm:Limit}}
generalizes another characterization}~{\cite{Xu2011}}\delled{ of joint
spectral radius.}}

\delxed{\begin{corollary}
\delled{If $\mu$ is of the form~{\eqref{eq:mu_s}} then}
\begin{align}
\stmathed{\hat \rho (\mathcal M) }
&\stmathed{=
\lim_{p\to\infty} \rho( E[A^{\otimes (2p)}] )^{1/(2p)}} \label{eq:Limit.2p.old}
\\
&\stmathed{=
\limsup_{p\to\infty} \rho( E[A^{\otimes p}] )^{1/p}.} \label{eq:LimitSup.old}
\end{align}
\end{corollary}}

\delxed{\begin{proof}
\delled{Let $K$ be the convex hull of $(\mathbb{R}^n)^{[2]}$, which is a
proper cone of $\mathbb{R}^{n_2}$ by Lemma~{\ref{lem:lifts}}. Then
$\supp(\mu^{[2]}) = (\supp \mu)^{[2]} $ clearly leaves $K$ invariant.
Also $\mu^{[2]}$ itself is a discrete measure consisting of finitely
many point masses. Therefore $\mu^{[2]}$ satisfies
{\ref{item:invariant}} and {\ref{item:mu.nice}}.}
\\
\delled{Now, {\eqref{eq:p-rad[m]}} shows}
\begin{equation}\label{eq:p-rad[2].old}
\stmathed{\hat{\rho}(\mathcal M) = \hat \rho(\mathcal M^{[2]})^{1/2}. }
\end{equation}
\delled{Also {\eqref{eq:p-rad.lift}} and Proposition~{\ref{prop:p-rad}}
yield $\rho_{p, \mu^{[2]}}^{1/2} =\rho(E[A^{\otimes (2p)}])^{1/(2p)}$.
Therefore Theorem~{\ref{thm:Limit}} shows $\hat \rho(\mathcal
M^{[2]})^{1/2} =\lim_{p\to\infty} \rho(E[A^{\otimes (2p)}])^{1/(2p)}$.
This identity and {\eqref{eq:p-rad[2].old}} prove
{\eqref{eq:Limit.2p.old}}. Then let us show {\eqref{eq:LimitSup.old}}.
Using the inequality~{\eqref{eq:p-rad.geq}}, the monotonicity of
$p$-radius (Lemma~{\ref{lem:basic.p-rad}}), and
Proposition~{\ref{prop:p-rad}}, we can show}
\begin{equation*}
\begin{aligned}
\stmathed{\rho(E[A^{\otimes (2p-1)}])^{1/(2p-1)}}
&\stmathed{\leq
\rho_{2p-1, \mu} }
\\
&\stmathed{\leq
\rho_{2p, \mu} }
\\
&\stmathed{=
\rho(E[A^{\otimes (2p)}])^{1/(2p)}. }
\end{aligned}
\end{equation*}
\delled{This inequality and {\eqref{eq:Limit.2p.old}} prove the
equation~{\eqref{eq:LimitSup.old}}.}
\end{proof}}

\section{Lyapunov Theorem for Switched Linear Systems\label{sec:ConvLypThm}}

As a theoretical application of \changed{Theorem~{\ref{thm:Limit}}}{the
limit formulas obtained in the last section}, in this section we show
a novel characterization of the absolute exponential stability of the
switched linear system~$\Sigma_{\mathcal M}$ \added{defined by
\eqref{eq:SigmaM}} via so-called stochastic Lyapunov functions. We
will also investigate the construction of stochastic Lyapunov
functions. Define the stochastic switched linear system~$\Sigma_\mu$
by
\begin{equation*}
\Sigma_{\mu}: x(k+1) = A_{k} x(k),\ 
\mbox{$A_k$ follows $\mu$ independently}
\end{equation*}
where $x(0) = x_0 \in \mathbb{R}^n$ is a constant.  We say that
$\Sigma_\mu$ is {\it exponentially stable in $p$th mean} ({\it $p$th
mean stable} for short) if there exist $C>0$ and~$\gamma \in [0, 1)$
such that \changed{$E[\norm{x(k)}^p] \leq C \gamma^{pk}\norm{x_0}^p$}{
\begin{equation*} 
E[\norm{x(k)}^p] \leq C \gamma^{pk}\norm{x_0}^p
\end{equation*}
} for every
$x_0 \in \mathbb{R}^n$. We call $\gamma$ as a {\it growth rate of the
$p$th mean}. As is expected, $p$th mean stability is closely related
to $p$-radius.

\begin{proposition}[\cite{Ogura2012b}]\label{prop:StblCond}
$\Sigma_\mu$ is $p$th mean stable if and only if $\rho_{p, \mu} < 1$.
\delled{Moreover the infimum of the growth rate of the $p$th mean equals
$\rho_{p, \mu}$.}
\end{proposition}

Now we introduce the notion of stochastic Lyapunov
functions\delled{~[19, 20, 21]} for $\Sigma_\mu$. The
\change{folloiwng}{following} definition extends the ones in
\cite{Bertram1959,Ahmadi1979,Feng1992} by allowing us to study $p$th
mean stability for a general $p$.
\begin{definition}
We say that $V\colon\mathbb{R}^n \to \mathbb{R}$ is a {\it stochastic
Lyapunov function of degree $p$} for~$\Sigma_\mu$ if there exist
positive numbers $C_1, C_2$\delled{ and $\gamma\in [0, 1)$} such that
\begin{equation}\label{eq:def.pLyap.norm}
C_1\norm{x}^p \leq V(x) \leq C_2 \norm{x}^p
\end{equation}
and \added{$\gamma \in [0, 1)$ such that}
\begin{equation}\label{eq:def.pLyap.mono}
E[V(Ax)] \leq \gamma^p V(x)
\end{equation}
for every $x\in \mathbb{R}^n$\added{, where $A$ is a random variable
following $\mu$}. We say that $V$ has a {\it growth rate}~$\gamma$.
\end{definition}

The next theorem is the main result of this section and provides a
connection between the stability of deterministic switched linear
systems and that of stochastic switched linear systems.

\begin{theorem}\label{thm:AASChar}
Let $\mu$ be a probability distribution satisfying
\ref{item:invariant} and \ref{item:mu.nice}. Let $\mathcal M = \supp
\mu$. Then $\Sigma_{\mathcal M}$ is absolutely exponentially stable if
and only if there exists $\gamma<1$ such that, for every $p\geq 1$,
$\Sigma_\mu$ admits a stochastic Lyapunov function of degree $p$ with
growth rate at most $\gamma$.
\end{theorem}

\begin{remark}
In contrast to the well known characterization of absolute exponential
stability with the existence of a single Lyapunov function called a
common Lyapunov function~\cite{Molchanov1989},
Theorem~\ref{thm:AASChar} characterizes absolute exponential stability
with the existence of {\it infinitely many} stochastic Lyapunov
functions. Also we notice that the above theorem deduces the existence
of Lyapunov functions from the absolute exponential stability and
hence can be considered as a version of converse Lyapunov
theorems~\cite{Molchanov1989,Dayawansa1999} in the systems and control
theory literature.
\end{remark}

For the proof of Theorem~\ref{thm:AASChar} we will need the next
proposition.

\begin{proposition}\label{prop:pLyap}
Let $\mu$ be a probability distribution on $\mathbb{R}^{n\times n}$.
\begin{enumerate}
\item If $\Sigma_\mu$ admits a stochastic Lyapunov function with
degree $p$ and growth rate~$\gamma<1$ then $\Sigma_\mu$ is $p$th mean
stable with growth rate~$\gamma$.

\item If $\Sigma_\mu$ is $p$th mean stable then, for every $\gamma \in
(\rho_{p,\mu}, 1)$, $\Sigma_\mu$ admits a stochastic Lyapunov function
with degree~$p$ and growth rate~$\gamma$.
\end{enumerate}
\end{proposition}

\begin{proof}
First assume that $\Sigma_\mu$ admits a stochastic Lyapunov
function~$V$ with degree~$p$ and growth rate~$\gamma<1$. Let $x_0 \in
\mathbb{R}^n$ be arbitrary. Using an induction we can show $E[V(x(k))]
\leq  \gamma^{pk} V(x_0)$. Therefore \eqref{eq:def.pLyap.norm} shows
$E[\norm{x(k)}^p] \leq (C_2/C_1) \gamma^{pk} \norm{x_0}^{p}$.  Thus
$\Sigma_\mu$ is $p$th mean stable with growth rate~$\gamma$.

On the other hand assume that $\Sigma_\mu$ is $p$th mean stable and
let $\gamma \in (\rho_{p, \mu} , 1)$ be arbitrary. We follow the
construction of Lyapunov functions in~\cite{Ando1998}. Define $h_k :=
E[\norm{A_k\dotsm A_1}^p]^{1/pk}$, where $A_1$, $A_2$, $\dotsc$ are
random variables following $\mu$ independently.  Since $h_k \to
\rho_{p, \mu}$ as $k\to\infty$, there exists $k_0$ such that $h_{k_0}
\leq \gamma$. Define $V\colon \mathbb{R}^n \to \mathbb{R}$ by
\begin{equation}\label{eq:def:nnorm}
V(x) := 
\sum_{k=0}^{k_0-1} \frac{E[\norm{A_k \dotsm A_1 x}^p]}{\gamma^{pk}}, 
\end{equation}
where, if $k=0$, the product $A_k\dotsm A_1$ is understood to be the
identity matrix with probability one. Let us show that $V$ is a
stochastic Lyapunov function for~$\Sigma_\mu$ with degree~$p$ and
growth rate~$\gamma$. It is immediate to see that $\norm{x}^p \leq
V(x) \leq [\sum_{k=0}^{k_0-1} {(h_k / \gamma)^{pk}} ]\norm{x}^p$,
where the first inequality can be obtained by truncating the series in
\eqref{eq:def:nnorm} at $k=0$. Moreover the independence of the random
variables $A_k$ yields
\begin{equation}\label{eq:E[V(Ax)]=...}
\begin{aligned}
E[{V(Ax)}]
&=
\sum_{k=0}^{k_0-1} \frac{E[\norm{A_{k+1} A_k \dotsm A_1 x}^p]}{\gamma^{pk}}
\\
&=
\gamma^p \sum_{k=1}^{k_0} \frac{E[\norm{A_k \dotsm A_1x}^p]}{\gamma^{pk}}.
\end{aligned}
\end{equation}
Since the last term of this sum can be bounded as
\begin{equation*}
\frac{E[\norm{A_{k_0} \dotsm A_1x}^p]}{\gamma^{{pk_0}}}
\leq
\frac{h_{k_0}^{{pk_0}}\norm{x}^p}{\gamma^{pk_0}}
\leq
\norm{x}^p, 
\end{equation*}
the equation~\eqref{eq:E[V(Ax)]=...} shows $E[V(Ax)] \leq \gamma^p
V(x)$. This completes the proof of the proposition.
\end{proof}

\begin{remark}
When $\mu$ is the uniform distribution on a finite set,
Proposition~\ref{prop:pLyap} \changed{can be seen to be true by just
taking the ($\epsilon$-)extremal norms of a finite set of matrices
studied in~{\cite{Protasov2008}}}{reduces to 
\cite[Proposition~1]{Protasov2008}, where for its proof the author makes use of
so-called extremal norms}.
\end{remark}

Now we prove Theorem~\ref{thm:AASChar}.

\begin{proof}[Proof of Theorem~\ref{thm:AASChar}]
Assume that $\Sigma_{\mathcal M}$ is absolutely exponentially stable.
Then there exist $C>0$ and $0\leq \gamma' < 1$ such that $\norm{x(k)}
< C\gamma'^k \norm{x_0}$ for any choice of the sequence
$\{A_k\}_{k=0}^\infty \subset \mathcal M$ and $x_0$. Now let $\gamma
\in (\gamma', 1)$ be arbitrary and let us fix $p\geq 1$. Since
$E[\norm{x(k)}^p] < C^p \gamma'^{pk}$, the system~$\Sigma_\mu$ is
$p$th mean stable with growth rate~$\gamma'$. Therefore, by
Proposition~\ref{prop:pLyap}, $\Sigma_\mu$ admits a stochastic
Lyapunov function with degree~$p$ and growth rate~$\gamma$.

On the other hand assume that there exists $\gamma <1$ such that, for
every $p\geq 1$, $\Sigma_\mu$ admits a stochastic Lyapunov function of
degree $p$ with growth rate $\gamma$. By Proposition~\ref{prop:pLyap},
$\Sigma_\mu$ is $p$th mean stable with growth rate $\gamma$, which
furthermore implies $\rho_{p, \mu} \leq \gamma$ by
Proposition~\ref{prop:StblCond}. Therefore $\lim_{p\to\infty} \rho_{p,
\mu} \leq \gamma <1$. Hence, by
\changed{Theorem~{\ref{thm:Limit}}}{Theorem~\ref{thm:main:equiv}}, we
obtain $\hat \rho(\mathcal M) < 1$, which gives the absolute
exponential stability of $\Sigma_{\mathcal M}$ by
Proposition~\ref{prop:AAS.char}.
\end{proof}

\subsection{Construction of Stochastic Lyapunov Functions}

The realization~\eqref{eq:def:nnorm} of a stochastic Lyapunov function
as a sum involving several expected values of products of matrices is
not useful in practice. In this section we will show that, if either
the conditions~\ref{item:p.even} or \ref{item:invariant} in
Proposition~\ref{prop:p-rad} holds, then we can construct stochastic
Lyapunov functions easily. 

The next theorem covers the case when \ref{item:p.even} holds.

\begin{theorem}\label{thm:lyapu.even}
Assume that $\Sigma_\mu$ is $p$th mean stable and \ref{item:p.even}
holds. Let $q := p/2$ and let $\gamma \in  (\rho_{p, \mu}, 1)$ be
arbitrary. Then the function
\begin{equation}\label{eq:V.even}
V(x) = (x^{\otimes q})^\top Hx^{\otimes q}, 
\end{equation} 
where the positive definite matrix~$H\in
\mathbb{R}^{n^{q}\times n^{q}}$ is a solution of the 
linear matrix inequality 
\begin{equation}\label{eq:LMI.even}
E[(A^{\otimes q})^\top H A^{\otimes q}] - \gamma^p H\preceq 0, 
\end{equation}
is a stochastic Lyapunov function for $\Sigma_\mu$ of degree $p$ with
growth rate $\gamma$.
\end{theorem}

To prove this theorem we need \changed{its special case}{the following
special case of the theorem given in \cite{Ogura2012b}}\delled{, which is
proved in~{\cite{Ogura2012b}}. We here quote the result for ease of
reference}.

\begin{proposition}\label{prop:squaremeanLyapunov}
Assume that $\Sigma_\mu$ is mean square stable. Let $\gamma \in
(\rho_{2,\mu}, 1)$ be arbitrary. Then the function~$V(x) = x^\top Hx$
on $\mathbb{R}^n$, where the positive definite matrix~$H \in
\mathbb{R}^{n\times n}$ is a solution of the linear matrix
inequality~$E[A^\top H A] - \gamma^2 H \preceq 0$, is a stochastic
Lyapunov function for $\Sigma_\mu$ with degree $2$ and growth rate
$\gamma$.
\end{proposition}

Let us prove Theorem~\ref{thm:lyapu.even}.

\begin{proof}[Proof of Theorem~\ref{thm:lyapu.even}]
Assume that $\Sigma_\mu$ is $p$th mean stable and let $\gamma \in
(\rho_{p, \mu}, 1)$ be arbitrary. Since \eqref{eq:p-rad.lift} and
Proposition~\ref{prop:StblCond} show $\rho_{2,\mu^{\otimes q}} =
\rho_{p, \mu}^{q} < 1$, the system~$\Sigma_{\mu^{\otimes q}}$ is mean
square stable again by Proposition~\ref{prop:StblCond}. Since
$\gamma^q > \rho_{p, \mu}^q = \rho_{2, \mu^{\otimes q}}$,
Proposition~\ref{prop:squaremeanLyapunov} implies that $\Sigma_\mu$
admits a stochastic Lyapunov function~$W(x) = (Hx, x)$ on
$\mathbb{R}^{n^q}$ with growth rate~$\gamma^q$. Moreover the
matrix~$H$ can be obtained as a solution of the matrix linear
inequality~$E[B^\top H B] - (\gamma^{q})^2 H \preceq 0$, where $B$ is
a random variable following $\mu^{\otimes q}$. This linear matrix
inequality is indeed equivalent to \eqref{eq:LMI.even}. Now define
$V\colon \mathbb{R}^n \to \mathbb{R}$ by \eqref{eq:V.even} or,
equivalently, by $V(x) := W(x^{\otimes q})$. Let us show that $V$ is a
stochastic Lyapunov function of degree~$p$ with growth rate~$\gamma$.
Using \eqref{eq:PowerKron} and \eqref{eq:KronExpct} we can see that
\delxed{\begin{equation} \label{eq:1Lyap->pLyap}
\begin{aligned}
\stmathed{E[V(Ax)]}
&\stmathed{=
E[W(A^{\otimes q}x^{\otimes q})]}
\\
&\stmathed{=
E[W(B x^{\otimes q})]}
\\
&\stmathed{\leq
(\gamma^q)^2 W(x^{\otimes q})}
\\
&\stmathed{=
\gamma^p V(x). }
\end{aligned}
\end{equation}}
\added{\begin{equation*}
\begin{aligned}
E[V(Ax)]
&=
E[W(A^{\otimes q}x^{\otimes q})]
\\
&=
E[W(B x^{\otimes q})]
\\
&\leq
(\gamma^q)^2 W(x^{\otimes q})
\\
&=
\gamma^p V(x). 
\end{aligned}
\end{equation*}}
To show that an inequality of the form \eqref{eq:def.pLyap.norm} holds
for $V$, notice that there exist positive constants $C_1, C_2$
satisfying $C_1\norm{y}^2 \leq W(y) \leq C_2 \norm{y}^2$ for every
$y\in\mathbb{R}^{n_q}$ because $H$ is positive definite. \changed{From
this inequality we can show}{Letting $y = x^{\otimes q}$ we obtain}
\eqref{eq:def.pLyap.norm} \changed{because}{by the well-known identity}
$\norm{x^{\otimes q}} = \norm{x}^q$ \added{(see, e.g.,
\cite{Brewer1978}) that holds} for a general $q$ and $x\in
\mathbb{R}^n$ {provided $\norm{\cdot}$ \changed{denotes}{is} the
Euclidean norm}. Hence $V$ is a stochastic Lyapunov function of degree
$p$ with growth rate~$\gamma$.
\end{proof}

Then we consider the condition~\ref{item:invariant}. In order to
proceed we \changed{will need the next basic lemma}{here quote a basic
result on $K$-positive matrices from~\cite{Vandergraft1968}}.

\begin{lemma}[\changed{[33, 25]}{\cite[Theorem~4.4]{Vandergraft1968}}]\label{lem:basic.cone}
Let $K$ be a proper cone \added{and assume $A >^K 0$}. \delled{If $A\geq^K
B\geq^K 0$ then $\rho(A)\geq \rho(B)$. Moreover if $A >^K 0$ then the
following statements are true.}
\begin{enumerate}
\item $A$ has a simple eigenvalue $\rho(A)$, which is greater than the
magnitude of any other eigenvalue of $A$;

\item The eigenvector corresponding to the eigenvalue $\rho(A)$ is in
$\interior K$.

\delxed{\item \delled{The dual cone~$K^*$ is a proper cone and $A^\top$ is
$K^*$-positive.}}
\end{enumerate} 
\end{lemma}

Then we prove the next proposition. Recall that, for a proper cone $K$
and $f\in \interior(K^*)$ the matrix norm $\norm{\cdot}_f$ is defined
by \eqref{eq:normA_f}\added{, which is induced by the cone linear
absolute norm $\norm{\cdot}_f$ satisfying \eqref{eq:ConeAbs} and
\eqref{eq:ConeLin}}.

\begin{proposition}\label{prop:rho(M)}
Let $K \subset \mathbb{R}^n$ be a proper cone and assume that $M \geq^K
0$. Also let $\epsilon>0$ be arbitrary.  Then there exists $f \in
\interior{(K^*)}$ such that  $\norm{M}_f < \rho(M)+\epsilon$.
\end{proposition}

\begin{proof}
\delled{Let $\epsilon>0$ be arbitrary. }First assume $M>^K 0$. By
Lemma~\ref{lem:basic.cone} the matrix~$M$ admits the Jordan canonical
form $J = V^{-1} M V$ where $V\in \mathbb{R}^{n\times n}$ is an
invertible matrix whose columns are the generalized eigenvectors of
$M$ and $J \in \mathbb{R}^{n\times n}$ is of the form
\begin{equation*}
J = \begin{bmatrix}
J_0 &  0 \\
0   & \rho(M )
\end{bmatrix}
\end{equation*}
for some upper diagonal matrix $J_0 \in \mathbb{R}^{(n-1)\times
(n-1)}$. Define $f \in \mathbb{R}^n$ by
\begin{equation*}
V^{-1} = \begin{bmatrix}
*\\f^\top
\end{bmatrix}. 
\end{equation*}
Then we can easily see that $f$ is an eigenvector of $M ^\top$
corresponding to the eigenvalue $\rho(M)$. Since $K^*$ is proper
\added{and $M^\top$ is $K^*$-positive (see, e.g., \cite{Berman1979})},
Lemma~\ref{lem:basic.cone} shows $f \in \interior(K^*)$.  Also since
$f^\top M x  = \rho(M)f^\top x$, the equation~\eqref{eq:norm_f} shows
$\norm{M }_f = \rho(M)$.

Then we consider the general case of $M \geq^K 0$. Let $\epsilon>0$ be
arbitrary and take an arbitrary $P >^K 0$. Then there exists
$\delta>0$ such that $\rho(M+\delta P) < \rho(M) + \epsilon$ because
$\rho(M+\delta P)\to \rho(M)$ as $\delta\to 0$ by the continuity of
spectral radius. Since $M+\delta P>^K 0$, the above argument shows
that there exists $f \in K^*$ satisfying $\norm{M +\delta P}_{f} =
\rho(M  + \delta P) < \rho(M) + \epsilon$. Finally, since $0\leq^K M
\leq^K M+\delta P$, Lemma~\ref{lem:SeidmanInducednorm} shows
$\norm{M}_{f} \leq \norm{M  + \delta P}_f$ and thus we obtain the
desired inequality.
\end{proof}

The next theorem enables us to construct a stochastic Lyapunov
function when \ref{item:invariant} holds.

\begin{theorem} \label{thm:lyapu.invariant}
Assume that $\Sigma_\mu$ is $p$th mean stable and \ref{item:invariant}
holds. Let $\gamma \in (\rho_{p, \mu}, 1)$ be arbitrary. Then there
exists a cone linear absolute norm $\norm{\cdot}$ on
$\mathbb{R}^{n^p}$ such that $V(x) = \norm{x^{\otimes p}}$ is a
stochastic Lyapunov function for $\Sigma_\mu$ with degree~$p$ and
growth rate~$\gamma$.
\end{theorem}

\begin{proof}
Assume that $\Sigma_\mu$ is $p$th mean stable and let $\gamma \in
(\rho_{p, \mu}, 1)$ be arbitrary. Let $K$ be a proper cone left
invariant by $\supp \mu$.

\newcommand{\fnorm}[1]{\norm{#1}_f}
\newcommand{\gnorm}[1]{\norm{#1}_g}

First we consider the special case $p=1$. Since $A$ leaves $K$
invariant with probability one we have $E[A] \geq^{K} 0$. Also
Proposition~\ref{prop:p-rad} shows \mbox{$\rho(E[A]) =\rho_{1,\mu} <
\gamma$}. Therefore, by Proposition~\ref{prop:rho(M)}, there exists a
cone linear absolute norm $\fnorm{\cdot}$ on $\mathbb{R}^n$ such that
$\fnorm{E[A]} < \gamma$. Let us show that $V(x) = \norm{x}_f$ gives a
stochastic Lyapunov function for $\Sigma_\mu$ with degree~$1$ and
growth rate~$\gamma$.

The inequality of the form \eqref{eq:def.pLyap.norm} clearly holds for
some positive constants~$C_1$ and~$C_2$ by the equivalence of the
norms on a finite dimensional normed vector space. To show
\eqref{eq:def.pLyap.mono} let $x\in\mathbb{R}^n$ and $\delta>0$ be
arbitrary. Since $\fnorm{\cdot}$ is cone linear absolute there exist
$x_1, x_2\in K$ such that $x=x_1-x_2$ and $\fnorm{x_1} + \fnorm{x_2} =
\fnorm{x_1+x_2} \leq \fnorm{x} +\delta$. Moreover we have $Ax_i \in K$
and therefore $\fnorm{Ax_i} = f^\top  {Ax_i}$ with probability one.
Thus it follows that
\begin{equation*}
\begin{aligned}
E[\fnorm{Ax_i}] 
&=
f^\top  E[A] x_i
\\
&=
\fnorm{E[A]x_i} 
\\
&<
\gamma \fnorm{x_i}
\end{aligned}
\end{equation*}
for each \changed{$i$}{$i = 1, 2$}. Hence, since $\fnorm{Ax} =
\fnorm{Ax_1 - Ax_2} \leq \fnorm{A x_1} + \fnorm{A x_2}$,
\begin{equation*}
\begin{aligned}
E[\fnorm{Ax}]
&<
\gamma (\fnorm{x_1} + \fnorm{x_2})
\\
&
\leq
\gamma(\fnorm{x} +\delta).
\end{aligned}
\end{equation*}
Since $\delta>0$ was arbitrary we obtain $E[\fnorm{Ax}]\leq
\gamma\fnorm{x}$. This inequality shows that, since $x\in
\mathbb{R}^n$ was arbitrary, $\fnorm{\cdot}$ is a stochastic Lyapunov
function for $\Sigma_\mu$ with growth rate~$\gamma$ and degree~$1$.

Then let us give the proof for a general $p$. Since $\rho_{1,
\mu^{\otimes p}} = \rho_{p, \mu}^p <1$ by \eqref{eq:p-rad.lift},
$\Sigma_{\mu^{\otimes p}}$ is first mean stable by
Proposition~\ref{prop:StblCond}.  Also notice that, by
Lemma~\ref{lem:lifts}, $\supp (\mu^{\otimes p}) = (\supp \mu)^{\otimes
p}$ leaves a proper cone in $\mathbb{R}^{n^p}$, say $K_p$, invariant.
Thus, by the above result for $p=1$, since $\gamma^p > \rho_{p, \mu}^p
= \rho_{1, \mu^{\otimes p}}$, the system~$\Sigma_{\mu^{\otimes p}}$
admits a stochastic Lyapunov function~$\gnorm{\cdot}$ with growth
rate~$\gamma^p$ and degree~$1$, where $\gnorm{\cdot}$ is a cone linear
absolute norm on $\mathbb{R}^{n^p}$ with respect~$K_p$. Now we define
$V \colon \mathbb{R}^n \to \mathbb{R}$ by \mbox{$V(x) :=
\gnorm{x^{\otimes p}}$}. Then, in the same way as
\changed{{\eqref{eq:1Lyap->pLyap}}}{the proof of
Theorem~\ref{thm:lyapu.even}}, we can show that $V$ is a stochastic
Lyapunov function for $\Sigma_\mu$ with degree~$p$ and growth
rate~$\gamma$.
\end{proof}

\begin{example}\label{ex:}
Consider the probability distribution 
\begin{equation*}
\mu = \begin{bmatrix}
[0, 1.5]& [0,1.8]\\
[0,0.15]&[0,1.2]
\end{bmatrix}, 
\end{equation*}
where each interval denotes the uniform distribution on it. Clearly
$\supp \mu$ leaves the proper cone~$\mathbb{R}^{2}_+$ invariant and
moreover we can see \mbox{$\rho(E[A]) < 1$}.  Therefore
Propositions~\ref{prop:StblCond} and~\ref{prop:p-rad} show that
$\Sigma_\mu$ is first mean stable and hence, by
Proposition~\ref{prop:pLyap}, $\Sigma_\mu$ admits a stochastic
Lyapunov function of degree~$1$. Following the proof of
Theorem~\ref{thm:lyapu.invariant} we find a stochastic Lyapunov
function~$\norm{x}_f$ for $\Sigma_\mu$ where $f = [0.3838\quad
1]^\top$. We generate $200$ sample paths of $\Sigma_\mu$ with the
initial state~$x_0 = [0\quad 1]^\top$.
\begin{figure}[tb]
\begin{center}
\includegraphics[height=6cm]{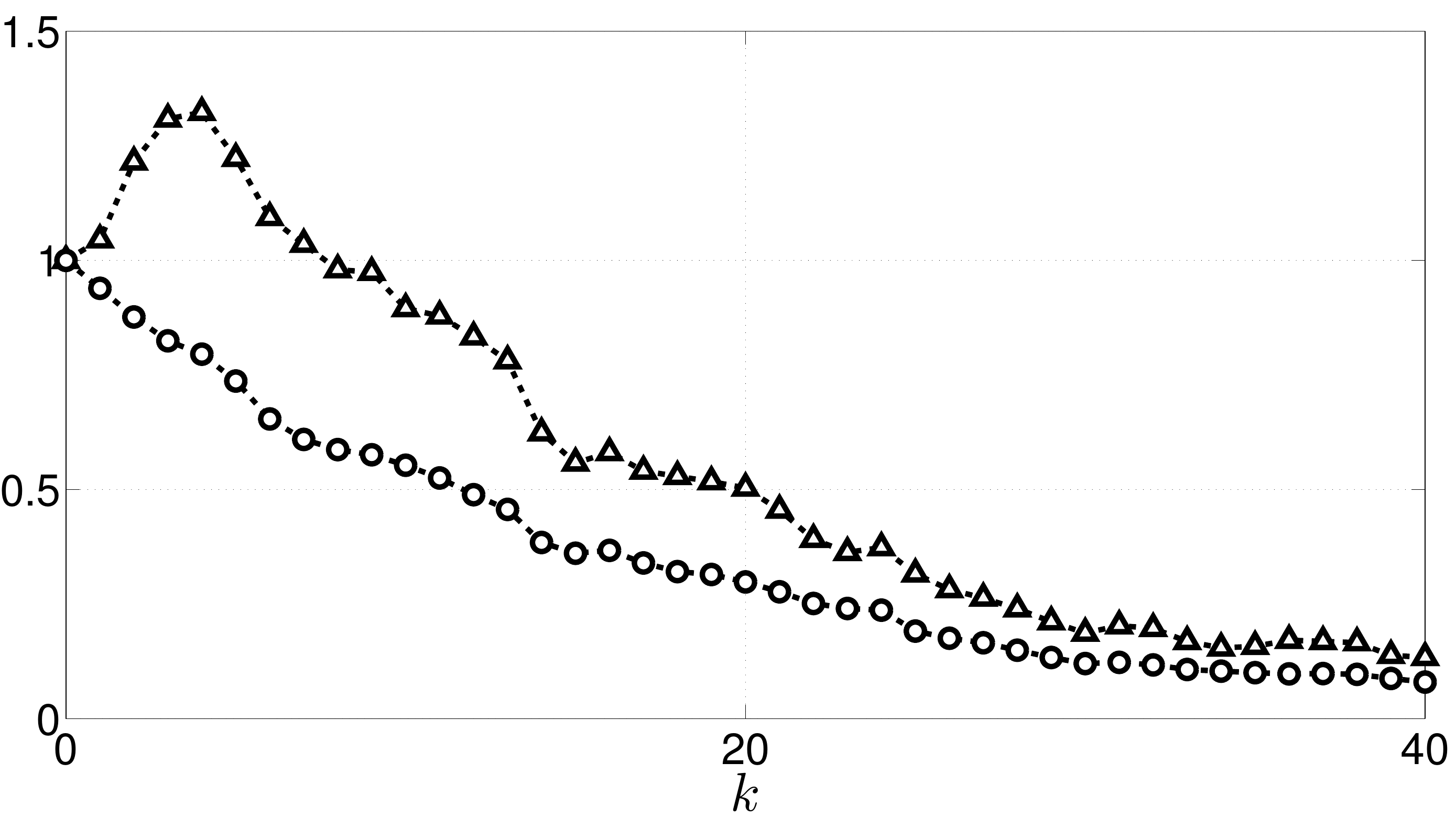}  \caption{The
sample means of the Lyapunov function (circle) and the Euclidean norm
(triangle)} \label{fig:Lyapunov}
\end{center}
\end{figure}
Figure~\ref{fig:Lyapunov} shows the sample means of the stochastic
Lyapunov function~$\norm{x(k)}_f$ and the Euclidean
norm~$\norm{x(k)}$. We can see that the \add{sample mean of} the
Lyapunov function decreases at the most of time instances, while
\add{that of} the Euclidean norm shows an oscillating behavior.
\add{It is remarked that the sample mean of the Lyapunov function is
not necessarily decreasing  because it is actually different from the
expectation. Taking more sample paths in general makes the sample mean
closer to the expectation by the law of large numbers and therefore is
more likely to yield a decreasing sample mean.}
\begin{figure}[tb]
\begin{center}
\includegraphics[height=6cm]{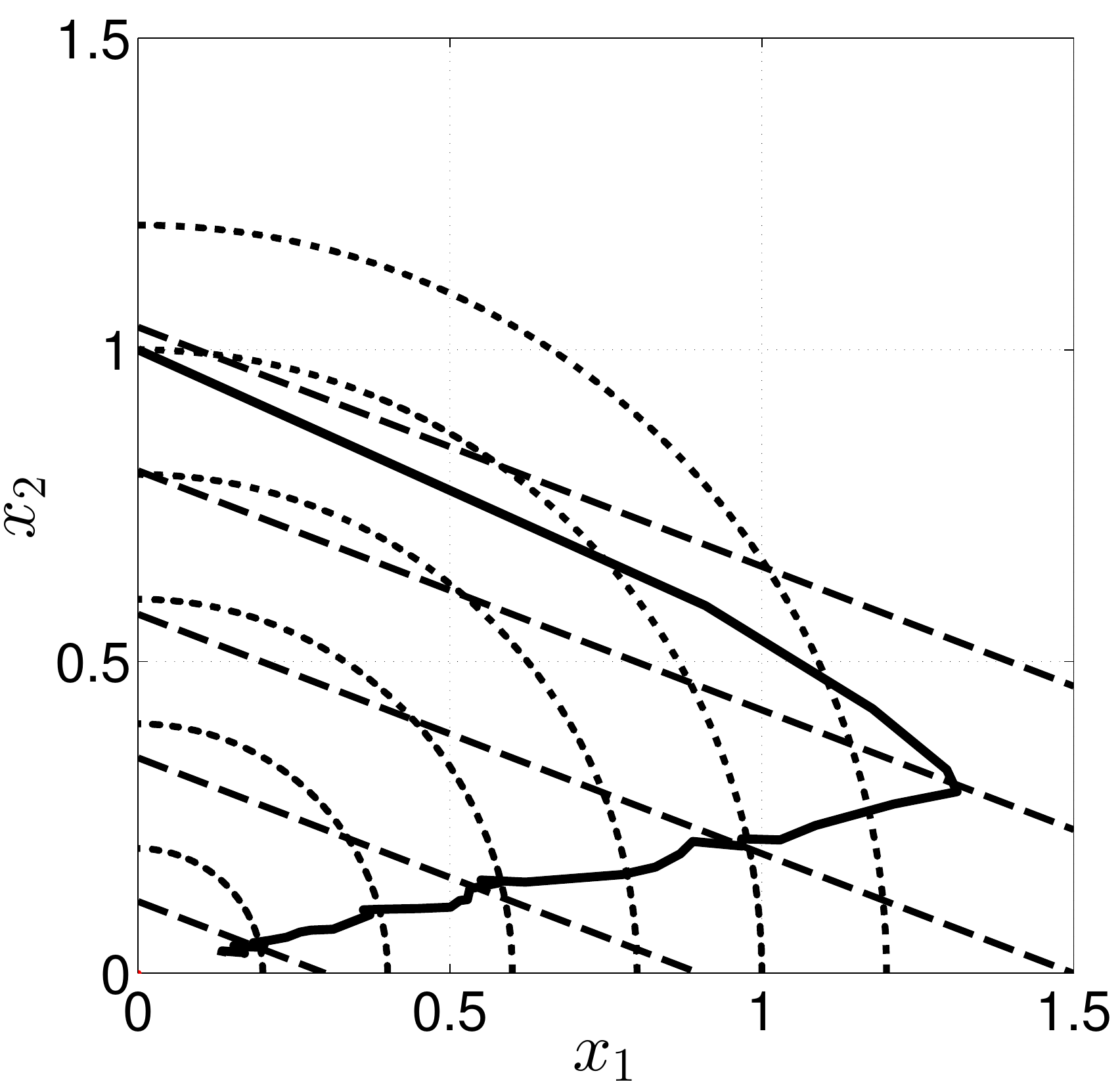}  \caption{The averaged
sample path (solid) and the level plots of the Lyapunov function
(dashed) and the Euclidean norm (dotted)} \label{fig:Level}
\end{center}
\end{figure}
Figure~\ref{fig:Level} shows the average of the sample paths and the
contour plot of the stochastic Lyapunov function and the Euclidean
norm. \added{The figure graphically illustrates that the sample mean is
evolving in such a way that the value of constructed Lyapunov function
almost decreases.}
\end{example}

\section{Conclusion and Discussion}\label{sec:concl}

This paper presented a characterization of the joint spectral radius
of a set of matrices as the limit of the $L^p$-norm joint spectral
radius of a probability distribution \added{on the set} when $p \to
\infty$ \added{under the assumption that the distribution has a certain
regularity and a support leaving a proper cone invariant}. The
obtained characterization extends the ones in the literature by
allowing the set to have infinitely many matrices. \changed{We made use
of nearly most unstable trajectories of an associated switched linear
system to derive the characterization.}{} Based on the result, we also
presented a novel characterization of the absolute exponential
stability of switched linear systems via the existence of stochastic
Lyapunov functions of any higher degrees. The construction of
stochastic Lyapunov functions is also studied.

\added{Understanding the behavior of $p$-radius as $p\to 0$ is one of
the problems closely related to the problem studied in this paper. It
is known~\cite{Fang2002} that, as $p \to 0$, the $p$-radius converges
to so-called Lyapunov exponent \cite{Arnold1984a} of random products
of matrices, which is known to characterize so-called almost sure
stability of stochastic switched systems. However the characterization
in \cite{Fang2002} is proved under the assumption that the number of
matrices in the set from which one takes a matrix is finite. It would
be interesting to investigate if one can allow the number of matrices
to be infinite with the approach taken in this paper. }





\ifdefined\vercomment
	\newpage
	
	\setcounter{page}{1}
	\newcommand\enumspacing{%
	\setlength{\parskip}{8pt}
	\setlength{\itemsep}{11pt}}
	
	\include{Response}

\else
\fi

\end{document}